\documentclass[12pt]{article}
\usepackage{amsmath,amssymb,amsbsy,amsfonts,amsthm,latexsym,
            amsopn,amstext,amsxtra,euscript,amscd,amsthm,url}

\usepackage[usenames, dvipsnames]{color}
\usepackage[utf8]{inputenc}
\usepackage[english]{babel}
\usepackage{mathrsfs}
\usepackage{caption}
\newtheorem{thm}{Theorem}
\newtheorem{lem}{Lemma}

\theoremstyle{definition}

\DeclareMathOperator{\1}{\textbf{1}}

\DeclareMathOperator{\id}{id}
\global\long\def\epsilon{\varepsilon}

\begin{document}
\date{}

\title{On sums of arithmetic functions involving the greatest common divisor}
\author{Isao Kiuchi, and Sumaia Saad Eddin}

\maketitle
{\def\thefootnote{}
\footnote{{\it Mathematics Subject Classification 2010: 11A25, 11N37, 11Y60.}}

\begin{abstract}
Let $\gcd(d_{1},\ldots,d_{k})$ be the greatest common divisor 
of the positive integers $d_{1},\ldots,d_{k}$, for any integer $k\geq 2$,
and let  $\tau$ and  $\mu$ denote the divisor function  
and  the M\"{o}bius function, respectively. For an arbitrary arithmetic function $g$ and for any real number $x>5$ and any integer $k\geq 3$, we define the sum
$$
S_{g,k}(x) :=\sum_{n\leq x}\sum_{d_{1}\cdots d_{k}=n} g(\gcd(d_{1},\ldots,d_{k}))
$$
In this paper, we give asymptotic formulas for $S_{\tau,k}(x)$ and $S_{\mu,k}(x)$ for $k\geq 3$.
\end{abstract}
\maketitle
%%%%%%%%%%%%%%%%%%%%%%%%%%%%%%%%%%%%%%%%%%%%%%%%%%%%%%%%%%%%%%%%%%%%%%%%%%%%%%%%%%%%%%%%%%%%%%%%%%%%%%%%%%%%%%%%%%%%%%%%%%%%%%%%%%%%%%%%%%%%%%%%%%%%%%%%%%%%%%%%%%%%%%%%%%%%%%%%%%%
\section{Introduction and main results}
Let $\gcd(d_{1},\ldots,d_{k})$ be  the greatest common divisor of the integers $d_{1},\ldots,d_{k}$ for any integer $k\geq 2$, and  let  
 $\mu$ and $\tau$ denote the M\"obius function and the divisor function, respectively. We recall that the Dirichlet convolution of two arithmetic  functions $f$ and $g$ is defined by  $f * g (n) = \sum_{d \mid n} f(d) g\left({n}/{d}\right)$  
for all positive integers $n$. 
The arithmetic functions $\1(n)$ and $\id (n)$ are defined by $\1(n) = 1$ 
and $\id (n)=n$ respectively. 
The function $\tau_k$ is the $k$-factors Piltz divisor function given by $\1*\1*\cdots \1$. In case of $k=2$, we write $\tau_2=\tau$. For an arbitrary arithmetic function $g$, we define the sum
\begin{align*}                                                        
f_{(g,k)}(n):= \sum_{d_{1}\cdots d_{k}=n}g\left(\gcd(d_{1},\ldots, d_{k})\right).
\end{align*}
In \cite{KNT}, Kr\"{a}tzel, Nowak and T\'{o}th gave asymptotic formulas for a class of arithmetic functions, which describe the value distribution of the greatest common divisor function. Typically, they are generated by a Dirichlet series whose analytic behavior is determined by the factor $\zeta^2(s)\zeta(2s-1),$ where $\zeta(s)$ is the Riemann zeta-function. In regards to the above formula, they proved that 
\begin{equation*}
 f_{(g,k)}(n) = \sum_{a^{k}b=n}\mu*g(a)\tau_{k}(b).
\end{equation*}
This identity was used to establish asymptotic formulas for $f_{(g,k)}(n)$ for specific choices of $g$ such as the identity function ${\rm id}$ and the sum of divisors function $\sigma={\rm id}*\1$. It was shown that
\begin{align*}                                                        
f_{(\id,k)}(n)=  \sum_{a^{k}b=n}\phi(a)\tau_{k}(b),
\quad        \quad                                                  
f_{(\sigma,k)}(n)=  \sum_{a^{k}b=n}a\tau_{k}(b),
\end{align*}
and that, for $\Re(s)>1$,  
\begin{align*}                                                        
\sum_{n=1}^{\infty} \frac{f_{(\id,k)}(n)}{n^s}=  \frac{\zeta^k(s)\zeta(ks-1)}{\zeta(ks)},
\quad        \quad                                                  
\sum_{n=1}^{\infty} \frac{f_{(\sigma,k)}(n)}{n^s}=  \zeta^k(s)\zeta(ks-1).
\end{align*}
In case of  $g=\delta :=\mu *\1$ (i.e., $\delta(1) = 1$ or $\delta(n) = 0$ else), the function 
$$
f_{(\delta,k)}(n)=\sum_{\substack{d_{1}\cdots d_{k}=n\\ \gcd(d_{1}\cdots d_{k})=1}}1
$$
has been considered in~\cite{I}.\\

Now, we define the sum
\begin{align*}                                                             
S_{g,k}(x) 
&:= \sum_{n\leq x} f_{(g,k)}(n) = \sum_{a^{k}b\leq x}\mu*g(a)\tau_{k}(b)
\end{align*} 
In two cases $g=\tau$ and $g=\mu$, it is easy to see that 
\begin{align}                                                              \label{tau-k}            
S_{\tau,k}(x) = \sum_{m\leq x^{\frac{1}{k}}}\sum_{\ell\leq x/m^k}\tau_{k}(\ell),
\end{align}
and that
\begin{align}                                                              \label{D-13}
S_{\mu,k}(x) 
&= \sum_{m\leq x^{\frac1k}}\mu*\mu(m)\sum_{n\leq \frac{x}{m^k}}\tau_{k}(n).
\end{align}
In this paper, we give asymptotic formulas for $S_{\tau,k}(x)$ and $S_{\mu,k}(x)$ for any integer $k\geq 3$. We have the following results. 
 \begin{thm}     
 \label{th11}
For any real number $x>5$, we have                                 
\begin{align}                                                                      \label{S-13}
& S_{\tau,3}(x)  
 =\frac{\zeta(3)}{2} x\log^{2}x + \zeta(3)\left(3\gamma -1 + 3 \frac{\zeta'(3)}{\zeta(3)}\right)x \log x    \\
& \ \ + \zeta(3) \left(3\gamma^{2} -3\gamma +3\gamma_{1}+1 
      +3(3\gamma -1)\frac{\zeta'(3)}{\zeta(3)}+\frac{9}{2}\frac{\zeta''(3)}{\zeta(3)}\right)x  
+ O\left(x^{\frac{43}{96}+\varepsilon}\right), \nonumber 
\end{align}
and                                                 
\begin{align}                                                              \label{S-14}          
& S_{\tau,4}(x)                                                         
 = \frac{\zeta(4)}{6} x\log^{3}x + \zeta(4)\left(2\gamma -\frac12 + 2\frac{\zeta'(4)}{\zeta(4)}\right) x\log^{2} x   \\
& +  \zeta(4) \left(6\gamma^{2} -4\gamma +4\gamma_{1}+1 
+8\left(2\gamma -\frac12\right)\frac{\zeta'(4)}{\zeta(4)}+ 8\frac{\zeta''(4)}{\zeta(4)}\right)x \log x \nonumber  \\
& +  \zeta(4)\left(12\gamma_{}\gamma_{1} + 4\gamma^{3} -6\gamma^{2}+4(\gamma -\gamma_{1}+\gamma_{2})-1 
+4(6\gamma^{2}-4\gamma + 4\gamma_{1}+1)\frac{\zeta'(4)}{\zeta(4)} \right)x   \nonumber  \\
&+16\zeta(4)\left(\left(2\gamma -\frac12\right) \frac{\zeta''(4)}{\zeta(4)}+ \frac{2}{3}\frac{\zeta^{(3)}(4)}{\zeta(4)}\right)x  
 + O\left({x^{\frac{1}{2}+\varepsilon}}\right),  \nonumber 
\end{align}
where $\gamma$ is the Euler constant, $\gamma_{1}$ and $\gamma_{2}$ 
are the Laurent-Stieltjes constants, (see Section~\ref{sec2} below for details). Here the function $\zeta^{(k)}(s)$ is the $k$-th derivative of the Riemann zeta-function $\zeta(s)$ with respect to $s$.  
\end{thm}
%%%%%%%%%%%%%%%%%%%%%%%%%%%%%%%%%%%%%%%%%%%%%%%%%%%%%%%%%%%%%%%%%%%%%%%%%%
%%%%%%%%%%%%%%%%%%%%%%%%%%%%%%%%%%%%%%%%%%%%%%%%%%%%%%%%%%%%%%%%%%%
 \begin{thm}      
 \label{th221}
Under the hypotheses of Theorem~\ref{th11} , we have
\begin{align}                                                           \label{S-23}  
&  S_{\mu,3}(x)                                                       
 =\frac{1}{2\zeta^{2}(3)} x \log^{2}x  
 + \frac{1}{\zeta^{2}(3)}\left(3\gamma -1 - 6 \frac{\zeta'(3)}{\zeta(3)}\right)x \log x    \\
& + \frac{1}{\zeta^{2}(3)} 
\left(3\gamma^{2} -3\gamma +3\gamma_{1}+1 -6(3\gamma -1) \frac{\zeta'(3)}{\zeta(3)} 
- 9\frac{\zeta''(3)}{\zeta(3)} + 27\left(\frac{\zeta'(3)}{\zeta(3)}\right)^{2}\right)x \nonumber \\
& \ \ \ + O\left(x^{\frac{43}{96}+\varepsilon}\right),    \nonumber  
\end{align}
and 
\begin{align} 
& S_{\mu,4}(x) = \frac{1}{6\zeta^{2}(4)} x\log^{3}x 
 + \frac{1}{\zeta^{2}(4)}\left(2\gamma -\frac12 - 4\frac{\zeta'(4)}{\zeta(4)}\right) x\log^{2} x             \label{S-24} \\
&+ \frac{16}{\zeta^{2}(4)}\left(\frac{6\gamma^{2} -4\gamma +4\gamma_{1}+1}{16} -\left(2\gamma -\frac12\right)\frac{\zeta'(4)}{\zeta(4)}
- \frac{\zeta''(4)}{\zeta(4)} +3\left(\frac{\zeta'(4)}{\zeta(4)}\right)^{2}\right) x \log x \nonumber \\
& +  \frac{1}{\zeta^2(4)}\left({12\gamma_{}\gamma_{1} + 4\gamma^{3} -6\gamma^{2}+4(\gamma -\gamma_{1}+\gamma_{2})-1} 
-8({6\gamma^{2}-4\gamma + 4\gamma_{1}+1}) \frac{\zeta'(4)}{\zeta(4)} \right)x  \nonumber \\
&+\frac{32}{\zeta^2(4)}\left(2\gamma -\frac12\right)\left(3\left(\frac{\zeta'(4)}{\zeta'(4)}\right)^2 - \frac{\zeta''(4)}{\zeta(4)}\right)x  \nonumber \\
&-\frac{64}{3\zeta^{2}(4)}\left(12\left(\frac{\zeta'(4)}{\zeta(4)}\right)^{3} -9 \frac{\zeta'(4)}{\zeta(4)}\cdot \frac{\zeta''(4)}{\zeta^{}(4)}
+  \frac{\zeta^{(3)}(4)}{\zeta(4)}\right) x
 + O\left({x^{\frac{1}{2}+\varepsilon}}\right). \nonumber 
\end{align}
\end{thm}
%%%%%%%%%%%%%%%%%%%%%%%%%%%%%%%%%%%%%%%%%%%%%%%%%%%%%%%%%%%%%%%%%%%%%%%%%%%%%%%%%%%%%%%%%%%%%%%%%%%%%%%%%%%%%%%%%%%%%%%%%%%%%%%%%%%%%%%%%%%%%%
The following theorem states asymptotic formulas of $S_{\tau,k}(x)$ and $S_{\mu,k}(x)$ for any integer $k\geq 5.$
 \begin{thm}     
 \label{th12}
 Let ${P}_{g,k}(u)$ be a polynomial in $u$ of degree $k-1$ depending on $g$. For $g=\tau$, $g=\mu$, and $k\geq 5$, we have    
\begin{align}                                                              \label{S-15}
S_{g,k}(x) = x {P}_{g,k}(\log x) + {E}_{g,k}(x)
\end{align}
where
\begin{align*}
&{E}_{g,k}(x)=
 O\left(x^{\frac{3k-4}{4k}+\varepsilon}\right) \quad  (5\leq k \leq 8),\\
 &{E}_{g,9}(x)=O\left(x^{\frac{35}{54}+\varepsilon}\right), \\
&{E}_{g,10}(x)=O\left(x^{\frac{41}{60}+\varepsilon}\right), \\
&{E}_{g,11}(x)=O\left(x^{\frac{7}{10}+\varepsilon}\right),\\
&{E}_{g,k}(x)=O\left(x^{\frac{k-2}{k+2}+\varepsilon}\right)  \quad \quad (12 \leq k \leq 25),\\
&{E}_{g,k}(x)=O\left(x^{\frac{k-1}{k+4}+\varepsilon}\right)  \quad \quad (26 \leq k \leq 50),\\
&{E}_{g,k}(x)=O\left(x^{\frac{31k-98}{32k}+\varepsilon}\right) \quad (51 \leq k \leq 57),\\
&{E}_{g,k}(x)=O\left(x^{\frac{7k-34}{7k}+\varepsilon}\right)  \ \quad (k\geq 58),
\end{align*}
for any small number $\varepsilon>0$. 
\end{thm}
%%%%%%%%%%%%%%%%%%%%%%%%%%%%%%%%%%%%%%%%%%%%%%%%%%%%%%%%%%%%%%%%%%%%%%%%%%%%%%%%%%%%%%%%%%%%%%%%%%%%%%%%%%%%%%%%%%%%%%%%%%%%%%%%%%%%%%%%%%%%%%%%%%%%%%
\section{Auxiliary results}
\label{sec2}
Before going into the proof of Theorems, we recall that the Laurent expansion of the Riemann zeta-function at its pole  $s=1$ is given by  
$$
\zeta(s)= \frac{1}{s-1}+ \sum_{k=0}^{\infty}\gamma_{k}(s-1)^{k}.
$$ 
Here the constants $\gamma_{k}$ are often called the Laurent-Stieltjes constants and it is known that 
\begin{equation*}
\gamma_n= \frac{(-1)^n}{n!}\lim\limits_{M\rightarrow \infty}\left(\sum\limits_{m=1}^{M}\frac{(\log m)^n }{m}-\frac{(\log M)^{n+1}}{(n+1)}\right),
\end{equation*}
for all $n\geq 0$ with $\gamma_0=\gamma=0,577\cdots$ being the Euler–Mascheroni constant.\\

We define  the error term $\Delta_{k}(x)$ of the Piltz divisor problem    by   
\begin{align}                                               
\label{piltz}
\Delta_{k}(x) := \sum_{n\leq x}\tau_{k}(n) - x~P_{k-1}(\log x),
\end{align} 
where $P_{k-1}(t)$ is a polynomial of degree $k-1$  in $t$.  
Notice that the coefficients of $P_{k-1}$ may be evaluated by using 
\begin{equation}
\label{eq2}
P_{k-1}(\log x) = \underset{s=1}{\rm Res}~\zeta^{k}(s)\frac{x^{s-1}}{s}.   
\end{equation}
From Eqs.~\eqref{piltz} and \eqref{eq2}, one can calculate explicity the coefficients of $P_{k-1}$ as functions of the Laurent-Stieltjes constants. For more details, see~\cite[Chapter 13]{I}.\\
%%%%%%%%%%%%%%%%%%%%%%%%%%%%%%%%%%%%%%%%%%%%%%%%%%%%%%%%%%%%%%

In order to prove our main results, it will be necessary to give some lemmas.
\begin{lem} \label{lem10}
We have 
\begin{align}                                                                            \label{d3}
\sum_{n\leq x}\tau_{3}(n) = \left(b_{1}\log^{2}x +b_{2}\log x + b_{3}\right) x  + \Delta_{3}(x),
\end{align}  
where 
$
\Delta_{3}(x) \ll x^{\frac{43}{96}+\varepsilon},  
$ 
and 
$$b_{1}=\frac12,\quad  b_{2}=3\gamma -1,  \quad
b_{3} = 3\gamma^{2}-3\gamma + 3\gamma_{1}+1. $$ 
Furthermore, we have      
\begin{align}                                                              \label{d4}
\sum_{n\leq x}\tau_{4}(n) =\left(c_{1}\log^{3}x + c_{2}\log^{2} x + c_{3} \log x   + c_{4}\right) x + \Delta_{4}(x)
\end{align}   
where 
$  
\Delta_{4}(x) \ll x^{\frac{1}{2}+\varepsilon},  
$   
and 
\begin{align*}
&c_{1}=\frac{1}{6},\quad  c_{2}=2\gamma -\frac{1}{2}, \quad c_{3}=6\gamma^{2}-4\gamma+4\gamma_{1}+1,\\
& c_{4}=12\gamma\gamma_{1}+4\gamma^{3}-6\gamma^{2}+ 4(\gamma - \gamma_{1}+\gamma_{2})-1.
\end{align*}
\end{lem}

\begin{proof} 
The proof of Eqs.~\eqref{d3} and \eqref{d4} can be found in \cite{K} and \cite[Theorems 13.2]{I}, respectively.  
\end{proof}
%%%%%%%%%%%%%%%%%%%%%%%%%%%%%%%%%%%%%%%%%%%%%%%%%%%%%%%%%%%%%%%%%%%%%%%%%%%%%%%%%%%%%%%%%%%%%%%%%%%%%%%%%%%%%%%%%%%%%%%%%%%%%%%%%%%%%%%%%%%%%%%%%%
\begin{lem} \label{lem15}
Let  $\alpha_{k}$ be the infimum of numbers $a_{k}$ such that    
$
\Delta_{k}(x)\ll \left(x^{a_{k}+\varepsilon}\right)
$ for any small $\varepsilon > 0$. Then

\begin{align*}
&\alpha_{k} \leq \frac{3k-4}{4k} \quad\quad   (5\leq k \leq 8),  \\
&\alpha_{9} \leq\frac{35}{54},    \\
&\alpha_{10} \leq\frac{41}{60}, \\  
&\alpha_{11} \leq\frac{7}{10},  \\ 
&\alpha_{k} \leq\frac{k-2}{k+2}   \quad \quad \quad  (12 \leq k \leq 25), \\ 
&\alpha_{k} \leq\frac{k-1}{k+4}   \quad \quad \quad (26 \leq k \leq 50),\\ 
&\alpha_{k} \leq\frac{31k-98}{32k} \quad \  (51 \leq k \leq 57),\\    
&\alpha_{k} \leq \frac{7k-34}{7k}   \quad  \quad  (k\geq 58).
\end{align*}
\end{lem}
\begin{proof} 
The proof of this lemma can be found in \cite[Theorem 13.2]{I}.  
\end{proof}
%%%%%%%%%%%%%%%%%%%%%%%%%%%%%%%%%%%%%%%%%%%%%%%%%%%%%%%%%%%%%%%%%%%%%%%%%%%%%%%%%%%%%%%%%%%%%%%%%%%%%%%%%%%%%%%%%%%%%%%%%%%%%%%%%%%%%%%%%%%%%%%%%%%%%%
\begin{lem}  \label{lem20}
For any  real number $x>5$ and any integer $k\geq 3$, we have 
\begin{align}  
\label{For-10}
\sum_{n\leq {x^{\frac1k}}}\frac{1}{n^k} 
&= \zeta(k) + \frac{1}{1-k}x^{\frac{1-k}{k}} +O\left(x^{-1}\right). 
\end{align}
Furthermore, we have  
\begin{align}
 \label{For-111}
&\sum_{n\leq {x^{\frac1k}}}\frac{\log^{}n}{n^k} 
= - \zeta^{'}(k) +\frac{1}{k(1-k)}x^{\frac{1-k}{k}}\log x -\frac{1}{(1-k)^2}x^{\frac{1-k}{k}} 
+  O\left(x^{-1}\log^{}x\right),                                            \\
&\sum_{n\leq {x^{\frac1k}}}\frac{\log^{2}n}{n^k} 
=  \zeta^{''}(k)  +\frac{1}{(1-k)k^2}x^{\frac{1-k}{k}}\log^{2}x   - \frac{2}{k(1-k)^2}x^{\frac{1-k}{k}}\log x   \nonumber \\
&\qquad  \qquad  \qquad \qquad    +  \frac{2}{(1-k)^3}x^{\frac{1-k}{k}} +O\left(x^{-1}\log^{2} x \right),                 \label{For-112} \\
&\sum_{n\leq {x^{\frac1k}}}\frac{\log^{3}n}{n^k} = -\zeta^{(3)}(k) 
 +\frac{1}{(1-k)k^3}x^{\frac{1-k}{k}}\log^{3}x   - \frac{3}{k^2 (1-k)^2}x^{\frac{1-k}{k}}\log^{2} x   \nonumber \\
& \qquad  \qquad + \frac{6}{k(1-k)^3}x^{\frac{1-k}{k}}\log x   +  \frac{6}{(1-k)^4}x^{\frac{1-k}{k}} + O\left(x^{-1}\log^{3} x \right).    \label{For-113}  
\end{align}
\end{lem} 
\begin{proof}
Eq.~\eqref{For-10} is given in \cite[Eq.~(14.40)]{I}. 
Let $z$ be any large real number, and let $r$ be any positive integer.  
Then we have 
\begin{align}                                                              \label{K}
\sum_{n\leq z}\frac{\log^{r}n}{n^s} = (-1)^{r}\zeta^{(r)}(s) 
 - \lim_{y\to \infty}\sum_{z<n\leq y}\frac{\log^{r}n}{n^s}
\end{align}
for any real number $s>1$. Using the fact that   
\begin{align*}
 \sum_{n\leq z}\log n = z\log z - z + O\left(\log z \right)    
\end{align*}
and the partial summation, we get 
\begin{align}  
\label{eq3}
\lim_{y\to \infty}\sum_{z<n\leq y}\frac{\log n}{n^s} = -\frac{1}{1-s}z^{1-s}\log z + \frac{1}{(1-s)^2}z^{1-s} +O\left(z^{-s}\log z\right).
\end{align}
Taking in the above formula $z=x^{1/k}$, $s=k$ and substituting Eq.~\eqref{eq3} into Eq.~\eqref{K} with $r=1$, we complete the proof of  Eq.~\eqref{For-111}.  \\

Similarly to the above, we use the following formula, \cite[Theorem 6.11]{N},
$$
\sum_{n\leq z}\log^{2} n = z\log^{2} z - 2z \log z + 2z + O\left(\log^{2} z \right) 
$$
and the partial summation to get 
\begin{align}  
\label{eq5}
\lim_{y\to \infty} \sum_{z<n\leq y}\frac{\log^{2}n}{n^s} &=  -\frac{1}{1-s}z^{1-s}\log^{2}z   \nonumber \\
&+ \frac{2}{(1-s)^2}z^{1-s}\log z   - \frac{2}{(1-s)^3}z^{1-s} +O\left(z^{-s}\log^{2} z \right).
\end{align}
Taking $z=x^{1/k}$ and $s=k$ in the latter formula and substituting Eq.~\eqref{eq5} into Eq.~\eqref{K} with  $r=2$, we get Eq.~\eqref{For-112}.
It remains to prove Eq.~\eqref{For-113}. Notice that
\begin{align*}
\sum_{n\leq z}\log^{3}n 
&= [z]\log^{3}z - 3\int_{1}^{z}\frac{[t]\log^{2}t}{t}dt  \\
&= z\log^{3}z - 3z \log^{2}z + 6z\log z -6z + O\left(\log^{3}z \right).
\end{align*}
Using the partial summation, we obtain that
\begin{align}  
\label{eq6}
\lim_{y\to \infty} \sum_{z<n\leq y}\frac{\log^{3}n}{n^s} 
&=  -\frac{1}{1-s}z^{1-s}\log^{3}z  + \frac{3}{(1-s)^2}z^{1-s}\log^{2} z  \nonumber \\
& - \frac{6}{(1-s)^3}z^{1-s}\log~{} z  + \frac{6}{(1-s)^4}z^{1-s} +O\left(z^{-s}\log^{3} z \right).
\end{align}
Combining Eqs.~\eqref{eq6} and \eqref{K} with $z=x^{1/k}$, $s=k$ and  $r=3$, we get the desired conclusion. 
\end{proof}
%%%%%%%%%%%%%%%%%%%%%%%%%%%%%%%%%%%%%%%%%%%%%%%%%%%%%%%%%%%%%%%%%%%%%%%%%%
Now, for any large real number $x>5$, we define
$$
M(x):=\sum_{n\leq x}\mu*\mu(n)
$$
It is known that $M(x)$ can be estimate by, see~\cite[Eq.~(4.11)]{I1}, 
\begin{align}                                                              \label{I-K} 
M(x) = O\left(x\varepsilon(x)\right),
\end{align}
where 
\begin{align}                                                              \label{eta}
\varepsilon(x) = {\rm exp}\left(-C \frac{(\log x)^{3/5}}{(\log\log x)^{1/5}}\right)
\end{align}
with  $C$ being a positive constant. Under the above hypotheses, we are ready to state the following result. 

%%%%%%%%%%%%%%%%%%%%%%%%%%%%%%%%%%%%%%%%%%%%%%%%%%%%%%%%%%%%%%%%%%%%%%%%%%%
\begin{lem}     
\label{lem30}
For any real number $x>5$ and any integers $k\geq 3$,  we have 
\begin{align} 
\label{mu-10}  
\sum_{n\leq {x^{\frac1k}}}\frac{\mu*\mu(n)}{n^k} 
&= \frac{1}{\zeta^{2}(k)} + O\left(x^{\frac{1-k}{k}}\varepsilon(x)\right), 
\end{align}
and, for $r$ any positive integer, 
\begin{align}
\label{mu-11}
\sum_{n\leq {x^{\frac1k}}}\frac{\mu*\mu(n)\log^{r}n}{n^k} 
&= M_{k,r}(k) +  O\left(x^{\frac{1-k}{k}}\varepsilon(x) \right)       
\end{align}
where  $\varepsilon(x)$ is given by \eqref{eta}. Here $M_{k,r}(k)$ are  certain constants depending on the Riemann zeta-function. 
Moreover, we have 
$$
M_{k,1}(k)=  2\frac{\zeta'(k)}{\zeta^{3}(k)}, \quad 
M_{k,2}(k)= 2\frac{3(\zeta'(k))^2 - \zeta''(k)\zeta(k)}{\zeta^{4}(k)}, 
$$
$$  %and \begin{align*}
  M_{k,3}(k)
=   \frac{2}{\zeta^{2}(k)}\left(12\left(\frac{\zeta'(k)}{\zeta(k)}\right)^{3} +3 \frac{\zeta'(k)}{\zeta(k)}\cdot \frac{\zeta''(k)}{\zeta^{}(k)}
-   \frac{\zeta^{(3)}(k)}{\zeta(k)}\right).  
$$ %$\end{align*}
\end{lem}

\begin{proof}
 We recall that, for $z$ any large real number and any  $s>1$,   
\begin{align}                                                                             \label{L}
\sum_{n\leq z}\frac{\mu*\mu(n)}{n^s} =  \frac{1}{\zeta^{2}(s)}    
                                    - \lim_{y\to \infty}\sum_{z<n\leq y}\frac{\mu*\mu(n)}{n^s}.
\end{align}
By Eq.~\eqref{I-K}, the partial summation and letting $y\to \infty$, we obtain  
\begin{align}                                                              \label{LL}
\lim_{y\to \infty}\sum_{z<n\leq y} \frac{\mu*\mu(n)}{n^s} 
= O\left(z^{1-s} \varepsilon(z)\right).
\end{align}
substituting \eqref{LL} into \eqref{L} with $z=x^{1/k}$ and $s=k$, we get Eq.~\eqref{mu-10}.  

Similarly to the above, we prove Eq.~\eqref{mu-11}. Notice that
\begin{align}                                                                             \label{N}
\sum_{n\leq z}\frac{\mu*\mu(n)\log^{r}n}{n^s} = \sum_{n=1}^{\infty}\frac{\mu*\mu(n)\log^{r} n}{n^s}    
                                    - \lim_{y\to \infty}\sum_{z<n\leq y}\frac{\mu*\mu(n) \log^{r} n}{n^s},
\end{align}
and that 
$$
\sum_{n=1}^{\infty}\frac{\mu*\mu(n) \log n}{n^s} = 2\frac{\zeta'(s)}{\zeta^{3}(s)}, 
\quad  
  \sum_{n=1}^{\infty}\frac{\mu*\mu(n) \log^2 n}{n^s}= 2\frac{3(\zeta'(s))^2 - \zeta''(s)\zeta(s)}{\zeta^{4}(s)},   
$$
$$
 \sum_{n=1}^{\infty}\frac{\mu*\mu(n) \log^3 n}{n^s} 
=   \frac{2}{\zeta^{2}(s)}\left(12\left(\frac{\zeta'(s)}{\zeta(s)}\right)^{3} +3 \frac{\zeta'(s)}{\zeta(s)}\cdot \frac{\zeta''(s)}{\zeta^{}(s)}
-   \frac{\zeta^{(3)}(s)}{\zeta(s)}\right).  
$$
Again by \eqref{I-K}  and using the partial summation, we find that  
\begin{align}                                                           \label{NN}
\lim_{y\to \infty}\sum_{z<n\leq y}\frac{\mu*\mu(n)\log^{r}n}{n^s} 
&= O\left(z^{1-s}\varepsilon(z) \log^{r}z\right)  
 = O\left(z^{1-s}\varepsilon(z) \right)                                                   
\end{align}
substituting \eqref{NN} into Eq.~\eqref{N} with $z=x^{1/k}$ and $s=k$,    Eq.~\eqref{mu-11} is proved.  
\end{proof}
%%%%%%%%%%%%%%%%%%%%%%%%%%%%%%%%%%%%%%%%%%%%%%%%%%%%%%%%%%%%%%%%%%%%%%%%%%%%%%%%%%%%%%%%%%%%%%%%%%%%%%%%%%%%%%%%%%%%%%%%%%%%%%%%%%%%%%%%%%%%%%%%%%%%%
\section{Proofs} 
\subsection{Proof of Theorem \ref{th11}}
\label{subsec1}
In case of $ k=3$, we substitute Eq.~\eqref{d3} into Eq.~\eqref{tau-k}  and then we use Eqs.~\eqref{For-10}, \eqref{For-111}, and \eqref{For-112} of Lemma~\ref{lem20} to obtain 
\begin{align*}                                                          
S_{\tau,3}(x) 
&=\left(b_{1}\log^{2}x + b_{2}\log^{}x  +b_{3} \right)x \sum_{n\leq {x^{\frac13}}}\frac{1}{n^3}   \\  
&- \left(6b_{1}\log^{} x + 3b_{2} \right) x \sum_{n\leq {x^{\frac13}}}\frac{\log n}{n^3} 
+ 9b_{1}x  \sum_{n\leq x^{\frac13}}\frac{\log^{2}n}{n^3} 
 + \sum_{n\leq {x^{\frac13}}}\Delta_{3}\left(\frac{x}{n^3}\right)   \\
&=\frac{\zeta(3)}{2}x\log^{2}x + \zeta(2)\left(3\gamma -1 +3\frac{\zeta'(3)}{\zeta(3)}\right)x\log x   \\
&+\zeta(3)\left(3\gamma^{2}-3\gamma + 3\gamma_{1}+1 +3(3\gamma-1)\frac{\zeta'(3)}{\zeta(3)}+\frac{9}{2}\frac{\zeta''(3)}{\zeta(3)}\right)x  \\
&- \frac{1}{8} \left(12\gamma^{2} -30\gamma +12\gamma_{1} +19\right)x^{1/3} 
+ \sum_{m\leq {x^{\frac13}}}\Delta_{3}\left(\frac{x}{n^3}\right)   + O\left(\log^{2}x\right).  
\end{align*}
Now, we use the estimate, see \cite{K},  
$$
\Delta_{3}(x) \ll x^{\frac{43}{96}+\varepsilon} 
$$
to get
\begin{align*}
     \sum_{n\leq x^{\frac13}} \Delta_{3}\left(\frac{x}{n^3}\right)
&\ll x^{\frac{43}{96}+\varepsilon} \sum_{n\leq x^{\frac13}}\frac{1}{n^{\frac{43}{32}}} 
\ll  x^{\frac{43}{96}+\varepsilon},
\end{align*} 
for any small number $\varepsilon>0$. This completes the proof of Eq.~\eqref{S-13}. 

In the case of $k=4$, we substitute  Eq.~\eqref{d4} into Eq.~\eqref{tau-k} 
\begin{align*}                                                          
& S_{\tau,4}(x)  
 =\left(c_{1}\log^{3}x + c_{2}\log^{2}x +c_{3}\log x +c_{4} \right)x \sum_{n\leq {x^{\frac14}}}\frac{1}{n^4}  \\  
&- 4\left(3c_{1}\log^{2} x + 2c_{2}\log x +c_{3}\right) x \sum_{n\leq {x^{\frac14}}}\frac{\log n}{n^4}  \\
&+ 4^{2}\left(3c_{1}\log x+ c_{2}\right)x  \sum_{n\leq x^{\frac14}}\frac{\log^{2}n}{n^4} 
 - 4^{3}c_{1}x \sum_{n\leq x^{\frac14}}\frac{\log^{3}n}{n^4}  +  \sum_{n\leq {x^{\frac14}}}\Delta_{4}\left(\frac{x}{n^4}\right),
 \end{align*}
 and use Lemma~\ref{lem20} to obtain 
\begin{align*} 
&S_{\tau,4}(x)  =\frac{\zeta(4)}{6} x\log^{3}x + \zeta(4) \left(2\gamma -\frac12 + 2\frac{\zeta'(4)}{\zeta(4)}\right)x \log^{2}x   \\  
& +  \zeta(4) \left(6\gamma^{2} -4\gamma +4\gamma_{1}+1 
+8\left(2\gamma -\frac12\right)\frac{\zeta'(4)}{\zeta(4)}+ 8\frac{\zeta''(4)}{\zeta(4)}\right)x \log x \nonumber  \\
& +  \zeta(4)\left(12\gamma_{}\gamma_{1} + 4\gamma^{3} -6\gamma^{2}+4(\gamma -\gamma_{1}+\gamma_{2})-1 
+4(6\gamma^{2}-4\gamma + 4\gamma_{1}+1)\frac{\zeta'(4)}{\zeta(4)} \right)x   \nonumber  \\
&+16\zeta(4)\left(\left(2\gamma -\frac12\right) \frac{\zeta''(4)}{\zeta(4)}+ \frac{2}{3}\frac{\zeta^{(3)}(4)}{\zeta(4)}\right)x  
  + \rho x^{\frac14} +  \sum_{m\leq {x^{\frac14}}}\Delta_{4}\left(\frac{x}{m^4}\right)  + O\left(\log^{3}x\right),   \nonumber 
 \end{align*}
with $\rho$  being a computable constant. 
By Lemma \ref{lem15}, we have
\begin{align*}
     \sum_{m\leq x^{\frac14}} \Delta_{4}\left(\frac{x}{m^4}\right)
&\ll x^{\frac{1}{2}+\varepsilon} \sum_{m\leq x^{\frac14}}\frac{1}{m^{2}} 
\ll  x^{\frac{1}{2}+\varepsilon}. 
\end{align*} 
Therefore Eq.~\eqref{S-14} is proved.          
%%%%%%%%%%%%%%%%%%%%%%%%%%%%%%%%%%%%%%%%%%%%%%%%%%%%%%%%%%%%%%%%%%%%%%%%%%%%%%%%%%%%%%%%%%%%%%%%%%%%%%%%%%%%%%%%%%%%%%%%%%%%%%%%%
\subsection{Proof of Theorem~\ref{th221}} 
In much the same way as in Subsection~\ref{subsec1}, we prove Eqs.~\eqref{S-23} and \eqref{S-24}. First, in the case of $k=3$, we substitute Eqs.~\eqref{d3}, \eqref{mu-10}, and \eqref{mu-11} into Eq.~\eqref{D-13} with $r=1,2$ to obtain  
\begin{align*}
&S_{\mu,3}(x)   
 =\frac{1}{2\zeta^{2}(3)} x \log^{2}x  
 + \frac{1}{\zeta^{2}(3)}\left(3\gamma -1 - 6 \frac{\zeta'(3)}{\zeta(3)}\right)x \log x    \\
& + \frac{1}{\zeta^{2}(3)} 
\left(3\gamma^{2} -3\gamma +3\gamma_{1}+1 -6(3\gamma -1) \frac{\zeta'(3)}{\zeta(3)} 
- 9\frac{\zeta''(3)}{\zeta(3)} + 27\left(\frac{\zeta'(3)}{\zeta(3)}\right)^{2}\right)x \\
&+ \sum_{n\leq {x^{\frac13}}}\mu*\mu(n)\Delta_{3}\left(\frac{x}{n^3}\right) + O\left(x^{1/3}\varepsilon(x)\right). 
\end{align*}
Again we use $\Delta_{3}(x)\ll x^{\frac{43}{96}+\varepsilon}$ to get
\begin{align*}
     \sum_{n\leq {x^{\frac13}}}\mu*\mu(n)\Delta_{3}\left(\frac{x}{n^3}\right)
&\ll x^{\frac{43}{96}+\varepsilon}\sum_{n\leq {x^{\frac13}}}\frac{\tau(n)}{n^{\frac{43}{32}}} 
\ll  x^{\frac{43}{96}+\varepsilon}. 
\end{align*} 
This completes the proof of Eq.~\eqref{S-23}.  

Second, for the case of $k=4$ and $r=1,2,3$, we substitute Eqs.~\eqref{d4}, \eqref{mu-10} and \eqref{mu-11} into Eq.~\eqref{D-13} to obtain 
\begin{align}                                                                  \label{33}
&S_{\mu,4}(x)   
=\frac{1}{6\zeta^{2}(4)} x\log^{3}x 
 + \frac{1}{\zeta^{2}(4)}\left(2\gamma -\frac12 - 4\frac{\zeta'(4)}{\zeta(4)}\right) x\log^{2} x               \\
&+ \frac{16}{\zeta^{2}(4)}\left(\frac{6\gamma^{2} -4\gamma +4\gamma_{1}+1}{16} -\left(2\gamma -\frac12\right)\frac{\zeta'(4)}{\zeta(4)}
- \frac{\zeta''(4)}{\zeta(4)} +3\left(\frac{\zeta'(4)}{\zeta(4)}\right)^{2}\right) x \log x \nonumber \\
& +  \frac{1}{\zeta^2(4)}\left({12\gamma_{}\gamma_{1} + 4\gamma^{3} -6\gamma^{2}+4(\gamma -\gamma_{1}+\gamma_{2})-1} 
-8({6\gamma^{2}-4\gamma + 4\gamma_{1}+1}) \frac{\zeta'(4)}{\zeta(4)} \right)x  \nonumber \\
&+\frac{32}{\zeta^2(4)}\left(2\gamma -\frac12\right)\left(3\left(\frac{\zeta'(4)}{\zeta'(4)}\right)^2 - \frac{\zeta''(4)}{\zeta(4)}\right)x  \nonumber \\
&-\frac{64}{3\zeta^{2}(4)}\left(12\left(\frac{\zeta'(4)}{\zeta(4)}\right)^{3} -9 \frac{\zeta'(4)}{\zeta(4)}\cdot \frac{\zeta''(4)}{\zeta^{}(4)}
+  \frac{\zeta^{(3)}(4)}{\zeta(4)}\right) x
+ \sum_{m\leq {x^{\frac14}}}\mu*\mu(m) \Delta_{4}\left(\frac{x}{m^4}\right)  \nonumber \\
& + O\left(x^{1/4}\varepsilon(x)\right). \nonumber 
\end{align}
Using $\Delta_{4}(x)\ll x^{\frac{1}{2}+\varepsilon}$, we get
\begin{align*}
     \sum_{m\leq {x^{\frac14}}}\mu*\mu(m) \Delta_{4}\left(\frac{x}{m^4}\right)
&\ll x^{\frac{1}{2}+\varepsilon}\sum_{m\leq {x^{\frac14}}}\frac{\tau(m)}{m^{2}} 
\ll  x^{\frac{1}{2}+\varepsilon}.
\end{align*} 
This completes the proof of Eq.~\eqref{S-24}.
%%%%%%%%%%%%%%%%%%%%%%%%%%%%%%%%%%%%%%%%%%%%%%%%%%%%%%%%%%%%%%%%%%%%%%%%%%%%%%%%%%%%%%%%%%%%%%%%%%%%%%%%%%%%%%%%%%%%%%%%%%%%%%%%%%%%%%%%%%%%%%%%%%%
\subsection{Proof of Theorem~\ref{th12}}
By the following formula and the partial summation
\begin{align}                                                      \label{log}
\sum_{n\leq z}\log^{r}z = z\log^{r}z + d_{1}z\log^{r-1}z +d_{2}z\log^{r-2}z + \cdots + d_{r}z + O\left(\log^{r}z\right),
\end{align} 
we find that  
$$
\lim_{y \to \infty}\sum_{z<n\leq y}\frac{\log^{r}n}{n^s} = - \frac{z^{1-s}}{1-s}\sum_{j=0}^{r}h_{j}(s) \log^{j}z,
$$
where $h_{j}(s)$ and $d_{j}$ $(1\leq j \leq r)$ are computable constants. By recalling Eq.~\eqref{K}
\begin{align*}                                                             
\sum_{n\leq z}\frac{\log^{r}n}{n^s} = (-1)^{r}\zeta^{(r)}(s) 
 - \lim_{y\to \infty}\sum_{z<n\leq y}\frac{\log^{r}n}{n^s}
\end{align*}
and taking $s=k$ and $z=x^{1/k}$, we deduce
\begin{align}                                                              \label{z}
\sum_{n\leq x^{1/k}}\frac{\log^{r}n}{n^k} 
&= (-1)^{r}\zeta^{(r)}(k) + \frac{x^{\frac{1-k}{k}}}{1-k}\sum_{j=0}^{r}\frac{h_{j}(k)}{k^{j}} \log^{j}x
\end{align} 
with $0\leq r \leq k$. Combining Eqs.~\eqref{tau-k}, \eqref{piltz} and \eqref{z}, we obtain 
\begin{align*}
S_{\tau,k}(x)  
&= \sum_{n\leq x^{1/k}}\left(\frac{x}{n^k} P_{k-1}\left(\log^{}\frac{x}{n^k}\right) + \Delta_{k}\left(\frac{x}{n^k}\right)\right)   \\
&= x\sum_{n\leq x^{1/k}}\frac{1}{n^k} \sum_{r=0}^{k-1} q_{r}\log^{r}\frac{x}{n^k} 
                  +\sum_{n\leq x^{1/k}} \Delta_{k}\left(\frac{x}{n^k}\right) \\
&=x\sum_{r=0}^{k-1}q_{r}\sum_{\ell=0}^{r}{r \choose \ell}(\log x)^{r-\ell} k^{\ell} \sum_{n\leq x^{1/k}}\frac{\log^{\ell}n}{n^k}    
                  +\sum_{n\leq x^{1/k}} \Delta_{k}\left(\frac{x}{n^k}\right) \\   
&=x\sum_{r=0}^{k-1}q_{r}\sum_{\ell=0}^{r} (-k)^{\ell} \zeta^{(\ell)}(k) {r \choose \ell}(\log x)^{r-\ell}    
                  +\sum_{e\leq x^{1/k}} \Delta_{k}\left(\frac{x}{n^k}\right) 
                  +O\left(x^{1/k}\log^{k}x\right) \\    
&= xP_{\tau,k}(\log x) + E_{\tau,k}(x),
\end{align*}
where
$$
E_{\tau,k}(x) = \sum_{n\leq x^{\frac1k}} \Delta_{k}\left(\frac{x}{n^k}\right) + O\left(x^{1/k}\log^{k}x\right),
$$
and $P_{\tau,k}(u)$ is a polynomial in $u$ of degree $k-1$ depending on  the derivative of the Riemann zeta-function.
From Lemma \ref{lem15} with $k\geq 5$, we have   
$$
\sum_{n\leq x^{\frac1k}} \Delta_{k}\left(\frac{x}{n^k}\right) \ll 
x^{\alpha_{k}+\varepsilon}\sum_{n\leq x^{1/k}}\frac{1}{n^{k\alpha_{k}+\varepsilon}} 
\ll   x^{\alpha_{k}+\varepsilon},
$$
where we used the fact that $k\alpha_{k}>2$. This completes the proof of  Theorem \ref{th12} in the case $g=\tau$. 

%%%%%%%%%%%%%%%%%%%%%%%%%%%%%%%%%%%%%%%%%%%%%%%%%%%%%%%%%%%%%%%%%%%%%%%%%%%
Similar arguments apply to the case $g=\mu.$ From  Eqs.~\eqref{D-13}, \eqref{mu-10} and \eqref{mu-11}, we get
\begin{align*}
&S_{\mu,k}(x)  
=x\sum_{r=0}^{k-1}q_{r}\sum_{\ell=0}^{r}{r \choose \ell}(\log x)^{r-\ell} k^{\ell} \sum_{n\leq x^{1/k}}\mu*\mu(n)\frac{\log^{\ell}n}{n^k}    
                  +\sum_{n\leq x^{1/k}}\mu*\mu(n) \Delta_{k}\left(\frac{x}{n^k}\right) \\ 
 &=x\sum_{r=0}^{k-1}q_{r}\sum_{\ell=0}^{r}{r \choose \ell}(\log x)^{r-\ell} k^{\ell}M_{k,\ell}(k) 
 +\sum_{m\leq x^{1/k}}\mu*\mu(m) \Delta_{k}\left(\frac{x}{m^k}\right)   + O\left(x^{\frac{1}{k}}\varepsilon(x)\right)\\   
&= xP_{\mu,k}(\log x) + E_{\mu,k}(x), 
\end{align*}
where  
$$
E_{\mu,k}(x) = \sum_{n\leq x^{\frac1k}} \mu*\mu(n) \Delta_{k}\left(\frac{x}{n^k}\right) + O\left(x^{1/k}\varepsilon(x)\right),
$$
and  $M_{k,0}(x)=1/\zeta^{2}(2)$.
From Lemma \ref{lem15}, the above sums can be estimated by
$$
x^{\alpha_{k}+\varepsilon}\sum_{n\leq x^{1/k}}\frac{\tau(n)}{n^{k\alpha_{k}+\varepsilon}} 
\ll   x^{\alpha_{k}+\varepsilon}.
$$
This completes the proof of our theorem.     
%%%%%%%%%%%%%%%%%%%%%%%%%%%%%%%%%%%%%%%%%%%%%%%%%%%%%%%%%%%%%%%%%%%%%%%%%%%%%%%%%%%%%%%%%%%%%%%%%%%%%%%%%%%%%%%%%%%%%%%%%%%%%%%%%%%%%%%%%%%%%%%%%

\section*{Acknowledgement}
The authors express their gratitude to Professor W{\l}adys{\l}aw Narkiewicz for carefully reading the paper and useful comments.
The second author is supported by the Austrian Science Fund (FWF): Projects F5507-N26 and F5505-N26, which are part
of the Special Research Program  ``Quasi Monte
Carlo Methods: Theory and Applications''.
%%%%%%%%%%%%%%%%%%%%%%%%%%%%%%%%%%%%%%%%%%%%%%%%%%%%%%%%%%%%%%%%%%%%%%%%%%%%%%%%%%%%%%%%%%%%%%%%%%%%%%%%%%%%%%%%%%%%%%%%%%%%%%%%%%%%%%%%%%%%%%%%%

\medskip\noindent {\footnotesize Isao Kiuchi: 
Department of Mathematical Sciences,
Faculty of Science, Yamaguchi University,
Yoshida 1677-1, Yamaguchi 753-8512, Japan.
e-mail: {\tt kiuchi@yamaguchi-u.ac.jp}}

\medskip\noindent {\footnotesize Sumaia Saad Eddin: 
Institute of Financial Mathematics and Applied Number Theory,
Johannes Kepler University,
Altenbergerstrasse 69, 4040 Linz, Austria.
e-mail: {\tt sumaia:saad\_eddin@jku.at}}

\end{document}